\documentclass [reqno]{amsart}
\usepackage{amsmath}
\usepackage{eurosym}
\usepackage{amsfonts}
\usepackage{amssymb}
\usepackage{amsmath}
\usepackage{amsfonts}
\usepackage{graphicx}
\usepackage{subfig}
\usepackage{amssymb}
\usepackage{amsmath}
\usepackage{amsfonts}
\usepackage{graphicx}
\usepackage{subfig}
\usepackage{amsfonts}

\newtheorem{theorem}{Theorem}
\theoremstyle{plain}

\newtheorem{corollary}{Corollary}

\newtheorem{definition}{Definition}

\numberwithin{equation}{section}

\begin{document}
\title[]{The investigation of some special curves in the pseudo-Galilean
4-space $G_{1}^{4} $}
\author{Fatma Almaz}
\address{department of mathematics, faculty of arts and sciences, batman
university, batman/ t\"{u}rk\.{ı}ye}
\email{fatma.almaz@batman.edu.tr}
\author{Handan \"{O}ztek\.{ı}n}
\address{department of mathematics, faculty of science, firat university,
elaz\i \u{g}/ t\"{u}rk\.{ı}ye}
\email{handanoztekin@gmail.com}
\thanks{This paper is in final form and no version of it will be submitted
for publication elsewhere.}

\begin{abstract}
In this paper, we investigate and characterize an arbitrary admissible curve in terms of its curvature functions in the pseudo-Galilean space $G_{1}^{4}$%
. We also give some special curves in four-dimensional pseudo-Galilean space and their characterizations, such as position curve, osculating curve, normal
curve, rectifying curve, slant helices, and spherical curves. Moreover, this study provides classifications of admissible curves in $G_{1}^{4}$. That is,
the admissible rectifying curves lying fully in the $G_{1}^{4}$ are expressed and given necessary and sufficient conditions for such a curve to be an admissible rectifying curve, osculating curve, normal curve, slant helices and spherical curves in $G_{1}^{4}$.
\end{abstract}

\maketitle


\section{Introduction}

An introduction to the fundamental concepts of curves in pseudo-Galilean
4-space will be given. $G_{1}^{4}$ space is a special type of geometry that
has a metric structure different from four-dimensional Euclidean space. The
study of curves in this space plays an important role in differential
geometry and related physical theories. It provides a foundation for further
studies of curves in $G_{1}^{4}$. The unique metric structure of this space
gives rise to properties in the behaviour of curves that differ from those
in Euclidean space.

Pseudo-Galilean space $G_{1}^{4}$ is important in physics, particularly in
the context of classical mechanics and Galilean relativity. Unlike the
space-time of special relativity, $G_{1}^{4}$ space preserves the concept of
absolute time. This means that time flows uniformly for observers in
different inertial reference frames. The fundamental importance of $%
G_{1}^{4} $ can be stated as follows:

As the mathematical framework for Galilean relativity, the space $G^{4}$
provides the mathematical foundation of Newtonian mechanics and Galilean
relativity. Galilean transformations are compatible with the principles of
classical physics, where velocities are simply added and time is universal.

If this space is to be expressed as a classical field theory, geometry or
related structures in $G_{1}^{4}$ can be used in some classical field
theories, particularly in fields such as super-fluidity and condensed matter
physics. In these fields, where the speed of light is finite, relativistic
effects can be neglected, and Galilean symmetry offers a more suitable
approach.

In physical terms, $G_{1}^{4}$ can be viewed as the limiting case of special
relativity, where the speed of light goes to infinity. This provides a
conceptual bridge to understanding how classical mechanics can be derived
from relativity.

The unique geometric structure of $G_{1}^{4}$, as expressed in studies of
Differential Geometry and particularly in mathematical physics, offers an
interesting avenue for the study of curves, surfaces, and other geometric
objects. These studies can form the basis for understanding more complex
space-time structures.

In summary, pseudo-Galilean space is a mathematical structure that embodies
absolute time and the simple velocity addition rule, forming the foundation
of classical physics and Galilean relativity. While Minkowski space is more
dominant in modern physics, $G_{1}^{4}$ remains a valuable tool for
understanding classical limits and modelling certain physical systems.

In classical differential geometry, the discovery of the moving
Frenet-Serret frame over regular curves paved the way for diverse studies on
different types of curves. Of course, with the help of ordinary differential
equations, researchers have investigated curves such as rectified curves,
normal curves, osculating spherical curves, involute-evolute curves,
Bertrand curves, helices, slant helices, and so on.

Given the general definitions and classifications of curves in
Pseudo-Galilean 4-space, let's briefly discuss some of the special and
important types of curves in this space. In summary, we'll briefly discuss
rectifying, osculating, normal, helical, and spherical curves. These curve
types are used in differential geometry to understand the local and global
properties of curves and while expressing these types of curves in $%
G_{1}^{4} $, the importance of taking into account the metric structure of
the space and the inner product derived from this metric cannot be denied.

Let's briefly provide some definitions about the types of curves we examine
in our study. Rectifying curves, for a curve to be rectifying means that its
tangent vector always lies within a fixed plane (or hyperplane). In $%
G_{1}^{4}$, if the span (linear extent) of a curve's tangent vector lies
within a fixed 2D subspace, that curve can be called rectifying. Osculating
curves, the term osculating plane is often used to describe the plane or
subspace closest to a curve. In classical differential geometry, the
osculating plane is formed by tangent and normal vectors. In $G_{1}^{4}$,
the osculating subspace of a curve is generally the subspace spanned by
tangent, normal, and bi-normal vectors. Osculating curves are curves for
which osculating subspaces have certain properties. Normal curves, the
normal of a curve generally refers to vectors perpendicular to the tangent
vector. In $G_{1}^{4}$, normal vectors are vectors perpendicular to the
tangent vector according to the metric in $G_{1}^{4}$. Normal curves can be
defined as curves in which all normal vectors remain within a fixed
subspace. A helix is a type of curve that rotates around a fixed axis.
In $G_{1}^{4}$, a helix can be defined as a curve in which the angle between
the tangent vector and a fixed directrix vector remains constant. Due to the
nature of the metric $G_{1}^{4}$, this concept of "angle" and the "fixed
directrix vector" can be interpreted in different ways, resulting in
different types of helices. Spherical curves are curves located on the
surface of a sphere around a fixed point. In $G_{1}^{4}$, the concept of a
"sphere" is defined as the set of points located at a fixed "distance" from
a fixed point according to metric given in $G_{1}^{4}$. This distance can be
space-like, time-like, or null, giving rise to different types of "spheres"
and, consequently, different types of spherical curves.

Examining these curves is important for understanding the geometry of the
space $G_{1}^{4}$ and the nature of motion in this space.

From the differential geometric point of view, the study of curves in $G_{3}$
has its own interest. Many interesting results on curves in $G_{3}$ and $%
G_{1}^{3}$ have been obtained by many authors \cite{10,11,15}. In \cite{1,20}, Frenet-Serret frame of a curve are construction by the authors in the
Galilean 4-space and pseudo Galilean 4-space and obtained the Frenet-Serret
equations. Some studies have been carried out different curves in Galilean
three and four spaces \cite{16,17,18,21}. Furthermore, more recent research
studies on position vectors in 3D and 4D Galilean spaces were performed
with the Frenet frame \cite{3,9,12}. In \cite{6}, the notion of the
involute-evolute curves in Minkowski 3-space $E_{1}^{3}$ are given by the
authors. In \cite{7}, the helices and slant helices are investigated using
non-degenerate curves in term of Sabban frame in de Sitter 3-space or Anti
de Sitter 3-space $M^{3}(\rho _{0})\subset $ $E_{2}^{4}$. In \cite{8}, the
representation formulas of non-null curves are primarily expressed and given
some certain results of describing the nun-null normal curve in four
dimensional semi-Euclidean space $E_{2}^{4}$. In \cite{2}, the position vector
of a time-like slant helix are expressed by the author in Minkowski 3-space.
In \cite{4,5}, the impact of magnetic fields on the moving particle
trajectories by variational approach to the magnetic flow associated with
the Killing magnetic field on lightlike cone is examined and different
magnetic curves are found in the 2D lightlike cone using the Killing
magnetic field of these curves. Also, some characterizations for $x$%
-magnetic curve and $x$-magnetic surface of rotation are given using the
Killing magnetic field of this curve in $Q^{2}$. In \cite{13}, the spacelike
rectifying curves lying fully in the 2-dimensional null cone and
3D null cone are characterized and given necessary and sufficient
conditions. In \cite{19}, Lorentzian indicators of curves with spacelike
principal normal are characterized by the author in Lorentzian 4-space $%
L^{4} $.

In this study, in the light of this paper, we extend aspects of classical
differential geometry topics to the pseudo-Galilean 4-Space by constructing Frenet-serret equations of an admissible curve and in the terms of this
frame, we express some special curves in 4D pseudo-Galilean space and their characterizations such as position curve, osculating curve, normal curve, rectifying curve, slant helices, spherical curves by using Frenet-Serret equations in the pseudo-Galilean 4-Space $G_{1}^{4}$.

\section{Preliminaries}

In affine coordinates the pseudo-Galilean scalar product between two points $%
P_{i}=(p_{i1},p_{i2},p_{i3},p_{i4}),i=1,2$ is defined by 
\begin{equation*}
g(P_{1},P_{2})=\{%
\begin{array}{c}
\left\vert p_{21}-p_{11}\right\vert , \\ 
\sqrt{\left\vert
-(p_{22}-p_{12})^{2}+(p_{23}-p_{13})^{2}+(p_{24}-p_{14})^{2}\right\vert },%
\end{array}%
\begin{array}{c}
\text{if }p_{21}\neq p_{11}, \\ 
\text{if \ }p_{21}=p_{11}.%
\end{array}%
\end{equation*}

We define the Pseudo-Galilean cross product for the vectors $\overrightarrow{%
u}$ $=(u_{1},u_{2},u_{3},u_{4}),$ $\overrightarrow{v}%
=(v_{1},v_{2},v_{3},v_{4})$ and $\overrightarrow{w}$ $%
=(w_{1},w_{2},w_{3},w_{4})$ as follows:

\begin{equation*}
\overrightarrow{u}\wedge \overrightarrow{v}\wedge \overrightarrow{w}%
=\left\vert 
\begin{array}{cccc}
0 & -e_{2} & e_{3} & e_{4} \\ 
u_{1} & u_{2} & u_{3} & u_{4} \\ 
v_{1} & v_{2} & v_{3} & v_{4} \\ 
w_{1} & w_{2} & w_{3} & w_{4}%
\end{array}%
\right\vert ,\text{ if }u_{1}\neq 0\text{ or }v_{1}\neq 0\text{ or }%
w_{1}\neq 0,
\end{equation*}%
\begin{equation*}
\overrightarrow{u}\wedge \overrightarrow{v}\wedge \overrightarrow{w}%
=\left\vert 
\begin{array}{cccc}
-e_{1} & e_{2} & e_{3} & e_{4} \\ 
u_{1} & u_{2} & u_{3} & u_{4} \\ 
v_{1} & v_{2} & v_{3} & v_{4} \\ 
w_{1} & w_{2} & w_{3} & w_{4}%
\end{array}%
\right\vert ,\text{ if }u_{1}=\text{ }v_{1}=w_{1}=0.
\end{equation*}

A vector $X(x,y,z,w)$ is called to be non-isotropic, if $x\neq 0$. All
unit non-isotropic vectors are of the form $(1,y,z,w)$. For isotropic
vectors, $x=0$ holds.

A non-lightlike vector is a unit vector if $-y^{2}+z^{2}+w^{2}=\pm 1.$

In the 4D pseudo-Galilean space $G_{4}^{1}$, there are isotropic vectors $%
X(x,y,z,w)$ and four types of isotropic vectors: spacelike ($x=0$, $%
-y^{2}+z^{2}+w^{2}>0$), timelike ($x=0$, $-y^{2}+z^{2}+w^{2}<0$) and two
types of lightlike vectors ($x=0$, $y=\sqrt{z^{2}+w^{2}}$ ). The scalar
product of two vectors $\overrightarrow{U}=(u_{1},u_{2},u_{3},u_{4})$ and $%
\overrightarrow{V}=(v_{1},v_{2},v_{3},v_{4})$ in $G_{4}^{1}$ is defined by

\begin{equation*}
\left\langle \overrightarrow{U},\overrightarrow{V}\right\rangle
_{G_{4}^{1}}=\{%
\begin{array}{c}
u_{1}v_{1},\text{ \ \ \ \ \ \ \ \ \ \ \ \ \ \ \ \ \ \ \ \ \ \ \ \ \ \ if }%
u_{1}\neq 0\text{ or }v_{1}\neq 0, \\ 
-u_{2}v_{2}+u_{3}v_{3}+u_{4}v_{4},\text{ \ if }u_{1}=0\text{ and }v_{1}=0.%
\end{array}%
\end{equation*}

The norm of vector $\overrightarrow{U}=(u_{1},u_{2},u_{3},u_{4})$ is defined
by 
\begin{equation*}
\left\Vert \overrightarrow{U}\right\Vert _{G_{4}^{1}}=\sqrt{\left\vert
\left\langle \overrightarrow{U},\overrightarrow{U}\right\rangle \right\vert
_{G_{4}^{1}}}.
\end{equation*}

\subsection{Construction of the Frenet-Serret Frame and the Frenet-Serret
Equations}

A curve $\alpha :I\subset 
\mathbb{R}
\rightarrow G_{4}^{1},$ $\alpha (t)=(x(t),y(t),z(t),w(t))$ is called an
admissible curve if $x^{\prime }(t)\neq 0.$

Let $\alpha :I\subset 
\mathbb{R}
\rightarrow G_{4}^{1},$ $\alpha (s)=(s,y(s),z(s),w(s))$ be an admissible
curve parametrized by arclength $s$ in $G_{4}^{1}.$ Here we denote
differentiation with respect to $s$ by a dash. The first vector of the
Frenet-Serret frame, that is the tangent vector of $\alpha $ is defined by%
\begin{equation*}
T=\alpha ^{\prime }(s)=(1,y^{\prime }(s),z^{\prime }(s),w^{\prime }(s)).
\end{equation*}

Since $T$ is a unit non-isotropic vector, so we can express%
\begin{equation}
\left\langle T,T\right\rangle _{G_{4}^{1}}=1.  \tag{2.1}
\end{equation}

Differentiating the formula (2.1) with respect to $s$, we have%
\begin{equation*}
\left\langle T^{\prime },T\right\rangle _{G_{4}^{1}}=0.
\end{equation*}

The vector function $T^{\prime }$ gives us the rotation measurement of the
curve $\alpha .$ The real valued function%
\begin{equation*}
\kappa (s)=\left\Vert T^{\prime }(s)\right\Vert =\sqrt{\left\vert
-(y^{\prime \prime }(s))^{2}+(z^{\prime \prime }(s))^{2}+(w^{\prime \prime
}(s))^{2}\right\vert }
\end{equation*}%
is called the first curvature of the curve $\alpha .$ We assume that, $%
\kappa (s)\neq 0,$ for all $s\in I.$ Similarly, we define
the principal normal vector%
\begin{equation*}
N(s)=\frac{T^{\prime }(s)}{\kappa (s)}
\end{equation*}
or another words 
\begin{equation}
N(s)=\frac{1}{\kappa (s)}(0,y^{\prime \prime }(s),z^{\prime \prime
}(s),w^{\prime \prime }(s)).  \tag{2.2}
\end{equation}

By the aid of the differentiation of the principal normal vector (2.2), we
define the second curvature function as%
\begin{equation}
\tau (s)=\left\Vert N^{\prime }(s)\right\Vert _{G_{4}^{1}}.  \tag{2.3}
\end{equation}%

This real valued function is called torsion of the curve $\alpha .$ The
third vector field, namely bi-normal vector field of the curve $\alpha $ is
defined by%
\begin{equation}
B_{1}(s)=\frac{1}{\tau (s)}(0,(\frac{y^{\prime \prime }(s)}{\kappa (s)}%
)^{\prime },(\frac{z^{\prime \prime }(s)}{\kappa (s)})^{\prime },(\frac{%
w^{\prime \prime }(s)}{\kappa (s)})^{\prime }).  \tag{2.4}
\end{equation}%

Thus, the vector $B_{1}(s)$ is both perpendicular to $T$ and $N.$ The fourth
unit vector is defined by%
\begin{equation}
B_{2}(s)=\mu T(s)\wedge N(s)\wedge B_{1}(s).  \tag{2.5}
\end{equation}%

Here, the coefficient $\mu $ is taken $\pm 1$ to make $+1$ determinant of the
matrix $\left[ T,N,B_{1},B_{2}\right] .$ We define the third curvature of
the curve $\alpha $ by the pseudo-Galilean inner product%
\begin{equation}
\sigma =\left\langle B_{1}^{\prime },B_{2}\right\rangle _{G_{4}^{1}}. 
\tag{2.6}
\end{equation}%

Here, as well known, the set $\left\{ T,N,B_{1},B_{2},\kappa ,\tau ,\sigma
\right\} $ is called the Frenet-Serret apparatus of the curve $\alpha .$ We
know that the vectors $\left\{ T,N,B_{1},B_{2}\right\} $ are mutually
orthogonal vectors satisfying%
\begin{equation}
\left\langle T,T\right\rangle _{G_{4}^{1}}=1,\left\langle N,N\right\rangle
_{G_{4}^{1}}=\varepsilon _{1},\left\langle B_{1},B_{1}\right\rangle
_{G_{4}^{1}}=\varepsilon _{2},\left\langle B_{2},B_{2}\right\rangle
_{G_{4}^{1}}=\varepsilon _{3},  \tag{2.7}
\end{equation}%
\begin{equation*}
\left\langle T,N\right\rangle _{G_{4}^{1}}=\left\langle T,B_{1}\right\rangle
_{G_{4}^{1}}=\left\langle T,B_{2}\right\rangle _{G_{4}^{1}}=\left\langle
N,B_{1}\right\rangle _{G_{4}^{1}}=\left\langle N,B_{2}\right\rangle
_{G_{4}^{1}}=\left\langle B_{2},B_{2}\right\rangle _{G_{4}^{1}}=0,
\end{equation*}%
where

\begin{equation*}
\varepsilon _{3}=\{%
\begin{array}{c}
+1,\text{ if }\varepsilon _{1}=-1\text{ or }\varepsilon _{2}=-1 \\ 
-1,\text{ if }\varepsilon _{1}=1\text{ and }\varepsilon _{2}=1%
\end{array}%
\end{equation*}
\ 

Now , we will obtain the Frenet-Serret equations for an admissible curve in $%
G_{4}^{1}.$

Let $\alpha (s)=(s,y(s),z(s),w(s))$ be an admissible curve parametrized by
arclength $s$ in $G_{4}^{1}$. Considering the above definitions, we have%
\begin{equation}
T^{\prime }=\varepsilon _{1}\kappa N,  \tag{2.8}
\end{equation}%
where $\varepsilon _{1}=\pm 1$. It is possible to define the vector $%
N^{\prime }$ according to frame $\left\{ T,N,B_{1},B_{2}\right\} $%
\begin{equation*}
N^{\prime }=\lambda _{1}T+\lambda _{2}N+\lambda _{3}B_{1}+\lambda _{4}B_{2},
\end{equation*}
$\lambda _{i}\in 
\mathbb{R}
,$ for $1\leq i\leq 4$. Multiplying both sides by the vectors $\left\{
T,N,B_{1},B_{2}\right\} $ and considering equation (2.1), we have,
respectively%
\begin{equation*}
\lambda _{1}=\lambda _{2}=0,\lambda _{3}=\varepsilon _{2}\tau ,\lambda
_{4}=0,
\end{equation*}
where $\varepsilon _{2}=-\varepsilon _{1}.$ Thus we have%
\begin{equation}
N^{\prime }=\varepsilon _{2}\tau B_{1}.  \tag{2.9}
\end{equation}

In order to compute the vector $B_{1}^{\prime },$ let us decompose%
\begin{equation*}
B_{1}^{\prime }=\lambda _{1}^{1}T+\lambda _{2}^{1}N+\lambda
_{3}^{1}B_{1}+\lambda _{4}^{1}B_{2},
\end{equation*}
where $\lambda _{i}^{1}\in 
\mathbb{R}
,$ for $1\leq i\leq 4$. Similar to $N^{\prime },$ we express%
\begin{equation*}
\lambda _{1}^{1}=0,\lambda _{2}^{1}=-\varepsilon _{2}\tau ,\lambda
_{3}^{1}=0,\lambda _{4}^{1}=\varepsilon _{3}\sigma .
\end{equation*}

So, we get%
\begin{equation*}
B_{1}^{\prime }=-\varepsilon _{2}\tau N+\varepsilon _{3}\sigma B_{2}.
\end{equation*}

Similar to above calculations, we can write%
\begin{equation*}
B_{2}^{\prime }=\lambda _{1}^{2}T+\lambda _{2}^{2}N+\lambda
_{3}^{2}B_{1}+\lambda _{4}^{2}B_{2},
\end{equation*}
and we obtain%
\begin{equation*}
\lambda _{1}^{2}=0,\lambda _{2}^{2}=0,\lambda _{3}^{2}=-\varepsilon
_{2}\sigma ,\lambda _{4}^{2}=0.
\end{equation*}

Thus we get%
\begin{equation}
B_{2}^{\prime }=-\varepsilon _{2}\sigma B_{1}.  \tag{2.10}
\end{equation}

Consequently we have following the Frenet-Serret equations in matrix form%
\begin{equation}
\frac{\partial }{\partial s}\left[ 
\begin{array}{c}
T \\ 
N \\ 
B_{1} \\ 
B_{2}%
\end{array}%
\right] =\left[ 
\begin{array}{cccc}
0 & \varepsilon _{1}\kappa & 0 & 0 \\ 
0 & 0 & \varepsilon _{2}\tau & 0 \\ 
0 & -\varepsilon _{2}\tau & 0 & \varepsilon _{3}\sigma \\ 
0 & 0 & -\varepsilon _{2}\sigma & 0%
\end{array}%
\right] \left[ 
\begin{array}{c}
T \\ 
N \\ 
B_{1} \\ 
B_{2}%
\end{array}%
\right] .  \tag{2.11}
\end{equation}

Four-dimensional pseudo-Galilean transformations, invariant properties under
translation and Minkowskian rotation, and the Frenet-Serret formulas
of a curve have been studied by \cite{1}. We will express the Frenet frame
as above and use the above expressions as references in next section. We
will also try to express the fundamental theorems of curve theory for some
special curves using this Frenet-Serret frame in $G_{1}^{4}$.

\section{Some characterizations of the special curves in $G_{1}^{4}$}

In this section, we consider an arbitrary admissible curve $\alpha :I\subset
IR\rightarrow G_{1}^{4}$ as a curve whose position vector satisfies the
parametric equation (3.1). Also, this study provides classifications of
admissible curves in $G_{1}^{4}$. That is, the admissible rectifying curve
lying fully in the $G_{1}^{4}$ are expressed and given necessary and
sufficient conditions for such curve to be a admissible rectifying curve in $%
G_{1}^{4}.$ Furthermore, by the classical differential geometry methods some
special curves are expressed in 4D pseudo-Galilean space and
their characterizations for example slant helix, 3-type slant helix,
spherical curves, osculating curves, rectifying curves, normal curves.

\begin{theorem}
The position vector of the admissible curve with curvatures $%
\kappa,\tau,\sigma \neq 0$, and with respect to the Frenet frame in the
pseudo-Galilean space $G_{1}^{4}$, it can be written as 
\begin{equation*}
\alpha (s)=\left( s+c^{\ast \ast }\right) T(s)-\left( \frac{\varepsilon _{2}%
}{\tau }\overset{..}{\mu }_{4}(t)+\frac{\sigma }{\tau }\mu _{4}(t)\right)
N(s)-\overset{.}{\mu }_{4}(t)B_{1}(s)+\mu _{4}(s)B_{2}(s)
\end{equation*}%
where 
\begin{equation*}
\mu _{4}(t)=t\int t\left( \int \frac{1}{t^{5}}e^{-\varepsilon _{2}\int 
\frac{\tau }{\overset{.}{\tau }}dt}\left( \frac{\varepsilon _{1}}{%
\varepsilon _{2}\varepsilon _{3}}\frac{\kappa \tau }{\sigma }\frac{h}{t}%
\int \frac{1}{t^{5}}e^{\varepsilon _{2}\int \frac{\tau }{\overset{.}{\tau }%
}dt}dt+c_{11}\right) dt\right) dt+c_{12},
\end{equation*}%
$c^{\ast \ast },c_{11},c_{12}\in 
\mathbb{R}
$, $t=\varepsilon _{3}\int \sigma (s)ds,$ and the following equations are
satisfied 
\begin{equation*}
3(\frac{1}{\tau }-\frac{\overset{.}{\tau }}{\tau ^{2}}t)+\frac{\sigma }{\tau 
}-\frac{1}{\varepsilon _{3}}\frac{\tau }{\sigma }=0\text{ ; }t\frac{d}{dt}(%
\frac{\sigma }{\tau })+\frac{\sigma }{\tau }-\frac{\varepsilon _{2}}{%
\varepsilon _{3}}\frac{\tau }{\sigma }=0.
\end{equation*}

\begin{proof}
Let $\alpha :I\subset IR\rightarrow G_{1}^{4}$ be an arbitrary admissible
curve in the pseudo-Galilean space $G_{1}^{4}$, then\ we can express its
position vector as%
\begin{equation}
\alpha (s)=\mu _{1}(s)T(s)+\mu _{2}(s)N(s)+\mu _{3}(s)B_{1}(s)+\mu
_{4}(s)B_{2}(s),  \tag{3.1}
\end{equation}%
for some differentiable functions, $\mu _{i}(s)$ and $1\leq i\leq 4$, by
differentiating this equation with respect to the arc-length parameter $s$
and using the Serret-Frenet equations (2.11), we obtain%
\begin{eqnarray*}
\alpha ^{\prime }(s) &=&\mu _{1}^{\prime }(s)T(s)+(\mu _{1}(s)\varepsilon
_{1}\kappa +\mu _{2}^{\prime }(s)-\mu _{3}(s)\varepsilon _{1}\tau )N(s) \\
&&+(\mu _{2}(s)\varepsilon _{2}\tau +\mu _{3}^{\prime }(s)-\mu
_{4}(s)\varepsilon _{2}\sigma )B_{1}(s)+(\mu _{3}(s)\varepsilon _{3}\sigma
+\mu _{4}^{\prime }(s))B_{2}(s)
\end{eqnarray*}%
it follows that 
\begin{equation}
\mu _{1}^{\prime }(s)=1  \tag{3.2a}
\end{equation}%
\begin{equation}
\mu _{1}(s)\varepsilon _{1}\kappa +\mu _{2}^{\prime }(s)-\mu
_{3}(s)\varepsilon _{2}\tau =0  \tag{3.2b}
\end{equation}%
\begin{equation}
\mu _{2}(s)\varepsilon _{2}\tau +\mu _{3}^{\prime }(s)-\mu
_{4}(s)\varepsilon _{2}\sigma =0  \tag{3.2c}
\end{equation}%
\begin{equation}
\mu _{3}(s)\varepsilon _{3}\sigma +\mu _{4}^{\prime }(s)=0.  \tag{3.2d}
\end{equation}%

From previous equations, we get%
\begin{equation}
\mu _{1}(s)=s+c^{\ast \ast }=h(t),  \tag{3.3}
\end{equation}%
we can use to change the variables to the variable $\frac{dt}{ds}%
=\varepsilon _{3}\sigma (s)$. Then, all functions of $s$ will transform to
the functions of $t$. Here, the derivative with respect to $t$ is denoted as
dot. Therefore, from Eq. (3.2), we get%
\begin{equation}
\mu _{3}(t)+\overset{.}{\mu }_{4}(t)=0,  \tag{3.4}
\end{equation}%
from the equation $\mu _{2}(s)\varepsilon _{2}\tau +\mu _{3}^{\prime
}(s)-\mu _{4}(s)\varepsilon _{2}\sigma =0$, it leads to%
\begin{equation}
\mu _{2}(t)=-\frac{\varepsilon _{2}}{\tau }\overset{..}{\mu }_{4}(t)-\frac{%
\sigma }{\tau }\mu _{4}(t),  \tag{3.5}
\end{equation}%
by taking derivative previous equation, we get 
\begin{equation}
\frac{d}{dt}\mu _{2}(t)=-\varepsilon _{2}\frac{d}{dt}(\frac{1}{\tau })%
\overset{..}{\mu }_{4}(t)-\frac{\varepsilon _{2}}{\tau }\overset{...}{\mu }%
_{4}(t)-\frac{d}{dt}(\frac{\sigma }{\tau })\mu _{4}(t)-\frac{\sigma }{\tau }%
\overset{.}{\mu }_{4}(t)  \tag{3.6}
\end{equation}
and from the equation $\mu _{1}(s)\varepsilon _{1}\kappa +\mu _{2}^{\prime
}(s)-\mu _{3}(s)\varepsilon _{2}\tau =0,$ it leads to%
\begin{equation*}
(s+c^{\ast \ast })\varepsilon _{1}\kappa +\overset{.}{\mu }%
_{2}(t)\varepsilon _{3}\sigma +\overset{.}{\mu }_{4}(t)\varepsilon _{2}\tau
=0
\end{equation*}%
and by calculating previous equation, we get 
\begin{equation}
\overset{.}{\mu }_{2}(t)=\frac{\varepsilon _{2}}{\varepsilon _{3}}\frac{\tau 
}{\sigma }\mu _{3}(t)-\frac{\varepsilon _{1}}{\varepsilon _{3}}\frac{\kappa 
}{\sigma }h(t)=-\frac{\varepsilon _{2}}{\varepsilon _{3}}\frac{\tau }{\sigma 
}\overset{.}{\mu }_{4}-\frac{\varepsilon _{1}}{\varepsilon _{3}}\frac{\kappa 
}{\sigma }h(t).  \tag{3.7}
\end{equation}%

When equations (3.6) and (3.7) are considered together, we get following
differential equation%
\begin{equation*}
\frac{\varepsilon _{2}}{\tau }\overset{...}{\mu }_{4}(t)-\varepsilon _{2}(%
\frac{\overset{.}{\tau }}{\tau ^{2}})\overset{..}{\mu }_{4}(t)+(\frac{\sigma 
}{\tau }-\frac{\varepsilon _{2}}{\varepsilon _{3}}\frac{\tau }{\sigma })%
\overset{.}{\mu }_{4}(t)+\frac{d}{dt}(\frac{\sigma }{\tau })\mu _{4}(t)=%
\frac{\varepsilon _{1}}{\varepsilon _{3}}\frac{\kappa }{\sigma }h(t).
\end{equation*}%

For the general solution of this equation, by using the condition 
\begin{equation*}
t\frac{d}{dt}(\frac{\sigma }{\tau })+\frac{\sigma }{\tau }-\frac{\varepsilon
_{2}}{\varepsilon _{3}}\frac{\tau }{\sigma }=0
\end{equation*}%
and by using transformation $\mu _{4}=tv$ the previous differential equation
is obtained as follows%
\begin{equation}
\overset{...}{v}(t)(\frac{\varepsilon _{2}}{\tau }t)-\overset{..}{v}(t)(3%
\frac{\varepsilon _{2}}{\tau }-\varepsilon _{2}\frac{\overset{.}{\tau }}{%
\tau ^{2}}t)+\overset{.}{v}(t)(-2\varepsilon _{2}\frac{\overset{.}{\tau }}{%
\tau ^{2}}+(\frac{\sigma }{\tau }-\frac{\varepsilon _{2}}{\varepsilon _{3}}%
\frac{\tau }{\sigma })t)=\frac{\varepsilon _{1}}{\varepsilon _{3}}\frac{%
\kappa }{\sigma }h(t),  \tag{3.8}
\end{equation}%
by making the transformation $\overset{.}{v}=w$ and by using following condition given as 
\begin{equation*}
3(\frac{1}{\tau }-\frac{\overset{.}{\tau }}{\tau ^{2}}t)+\frac{\sigma }{\tau 
}-\frac{1}{\varepsilon _{3}}\frac{\tau }{\sigma }=0.
\end{equation*}%

By making similar transformations in equation (3.8) is obtained as follows.%
\begin{equation}
\mu _{4}(t)=t\int t\left( \int \frac{1}{t^{5}}e^{-\varepsilon _{2}\int 
\frac{\tau }{\overset{.}{\tau }}dt}\left( \frac{\varepsilon _{1}}{%
\varepsilon _{2}\varepsilon _{3}}\frac{\kappa \tau h}{\sigma t}\int \frac{1%
}{t^{5}}e^{\varepsilon _{2}\int \frac{\tau }{\overset{.}{\tau }}%
dt}dt+c_{11}\right) dt\right) dt+c_{12},  \tag{3.9}
\end{equation}%
where $c_{11}$ and $c_{12}$ are arbitrary constants. From Eqs. (3.3), (3.4)
and (3.5), we obtain the functions $\mu _{3}(t),$ $\mu _{2}(t)$ as%
\begin{equation*}
\mu _{2}(t)=\frac{-\varepsilon _{2}}{\tau }\overset{..}{\mu }_{4}(t)-\frac{%
\sigma }{\tau }\mu _{4}(t);\mu _{3}(t)=-\overset{.}{\mu }_{4}(t),
\end{equation*}%
from Eqs. (3.9), the solution is obtained and thus, the proof is completed.
\end{proof}
\end{theorem}

\begin{theorem}
Let $\alpha $ be a unit speed curve in $G_{1}^{4}$ with the curvatures $%
\kappa ,\tau ,\sigma \neq 0$ and given as $t=\varepsilon _{3}\int \sigma
(s)ds.$ Then, the following expressions are provided.

1) the distance distance $d(t)=\left\Vert \alpha (t)\right\Vert $ satisfies 
\begin{equation*}
d^{2}(t)=2\int \mu _{1}(t)dt-2\int \mu _{2}(t)\left( \mu _{1}(t)\kappa
+(\varepsilon _{1}\varepsilon _{2}-1)\mu _{3}(t)\tau \right) dt,
\end{equation*}%
and If $\alpha $ is a rectifying curve the distance distance is given as
\begin{equation*}
d^{2}(t)=2\int \mu _{1}(t)dt=s^{2}+cs+d
\end{equation*}%
where $\mu _{1}(s)=s+c^{\ast \ast }$, $\mu _{2}(t)=\frac{-\varepsilon _{2}}{%
\tau }\overset{..}{\mu }_{4}(t)-\frac{\sigma }{\tau }\mu _{4}(t),$ $\mu
_{3}(t)=-\overset{.}{\mu }_{4}(t)$, $c,d\in 
\mathbb{R}
.$

2) the normal component $\alpha ^{N}(s)$ of the curve $\alpha$ is given as
\begin{equation*}
\left\Vert \alpha ^{N}\right\Vert =d^{2}(t)-\mu _{1}^{2}(t)
\end{equation*}%
and If $\alpha $ is a rectifying curve the normal component $\alpha ^{N}$ is given as
\begin{equation*}
\left\Vert \alpha ^{N}\right\Vert =2\int \mu _{1}(t)dt-\mu _{1}^{2}(t).
\end{equation*}

\begin{proof}
Let $\alpha $ be the curve with unit speed and non-zero curvatures $\kappa
,\tau ,\sigma \neq 0$ in $G_{1}^{4}$. The position vector of the curve $%
\alpha $ satisfies the equation (3.1). Also, let the functions $\mu
_{1}(t),\mu _{2}(t),\mu _{3}(t),\mu _{4}(t)$ be differentiable functions
satisfying the equation (3.2). Hence, the following equation can be written 
\begin{equation*}
d(t)=\mu _{1}^{2}(t)+\varepsilon _{1}\mu _{2}^{2}(t)+\varepsilon _{2}\mu
_{3}^{2}(t)+\varepsilon _{3}\mu _{4}^{2}(t)
\end{equation*}%
and in the equation system (3.2), when the equations (3.2a), (3.2b), (3.2c) and (3.2d) are multiplied by $\mu _{1}(t),\varepsilon _{1}\mu
_{2}(t),\varepsilon _{2}\mu _{3}(t),\varepsilon _{3}\mu _{4}(t)$
respectively and added together, we get 
\begin{equation}
\mu _{1}(t)\mu _{1}^{\prime }(s)=\mu _{1}(t)  \tag{3.10a}
\end{equation}%
\begin{equation}
\mu _{2}(t)\mu _{1}(s)\kappa +\varepsilon _{1}\mu _{2}(t)\mu _{2}^{\prime
}(s)-\varepsilon _{1}\varepsilon _{2}\mu _{2}(t)\mu _{3}(s)\tau =0 
\tag{3.10b}
\end{equation}%
\begin{equation}
\mu _{3}(t)\mu _{2}(s)\tau +\varepsilon _{2}\mu _{3}(t)\mu _{3}^{\prime
}(s)-\mu _{3}(t)\mu _{4}(s)\sigma =0  \tag{3.10c}
\end{equation}%
\begin{equation}
\mu _{4}(t)\mu _{3}(s)\sigma +\varepsilon _{3}\mu _{4}(t)\mu _{4}^{\prime
}(s)=0  \tag{3.10d}
\end{equation}%
%
and%
\begin{eqnarray*}
&&\mu _{1}(t)\mu _{1}^{\prime }(s)+\varepsilon _{1}\mu _{2}(t)\mu
_{2}^{\prime }(s)+\varepsilon _{2}\mu _{3}(t)\mu _{3}^{\prime }(s) \\
+\varepsilon _{3}\mu _{4}(t)\mu _{4}^{\prime }(s) &=&\mu _{1}(t)+\mu
_{2}(t)(-\mu _{1}(s)\kappa +(\varepsilon _{1}\varepsilon _{2}-1)\mu
_{3}(t)\tau )
\end{eqnarray*}%
and when multiplied by two and the derivative is taken, we have%
\begin{eqnarray*}
d^{2}(t) &=&\mu _{1}^{2}(t)+\varepsilon _{1}\mu _{2}^{2}(t)+\varepsilon
_{2}\mu _{3}^{2}(t)+\varepsilon _{3}\mu _{4}^{2}(t) \\
&=&2\int (\mu _{1}(t)+\mu _{2}(t)(-\mu _{1}(s)\kappa +(\varepsilon
_{1}\varepsilon _{2}-1)\mu _{3}(t)\tau )dt.
\end{eqnarray*}%

If the curve $\alpha $ is a rectifying curve, the position vector of the
curve $\alpha $ must be $\left\langle \alpha (t),N(t)\right\rangle =\mu
_{2}(t)=0$. Therefore, considering the above expression and the equations $%
\mu _{1}(s)=s+c^{\ast \ast }=h(t),\frac{dt}{ds}=\varepsilon _{3}\sigma $, we
obtain the following equation 
\begin{equation}
d^{2}(t)=2\int \mu _{1}(t)dt=s^{2}+cs+d;c,d\in 
\mathbb{R}
.  \tag{3.11}
\end{equation}%

Moreover, the normal component of the position vector $\alpha ^{N}$ of the
curve $\alpha $ is found as follows 
\begin{equation*}
\left\Vert \alpha ^{N}\right\Vert =d^{2}(t)-\mu _{1}^{2}(t)=2\int (\mu
_{1}(t)+\mu _{2}(t)(-\mu _{1}(s)\kappa +(\varepsilon _{1}\varepsilon
_{2}-1)\mu _{3}(t)\tau )dt-\mu _{1}^{2}(t).
\end{equation*}%

If $\alpha $ is a rectifying curve, since $\mu _{2}(t)=0$ the last equation
is found as%
\begin{equation}
\left\Vert \alpha ^{N}\right\Vert =2\int \mu _{1}(t)dt-\mu _{1}^{2}(t). 
\tag{3.12}
\end{equation}
\end{proof}
\end{theorem}

\begin{theorem}
Let $\alpha :I\subset IR\rightarrow G_{1}^{4}$ be an arbitrary admissible
curve $\kappa ,\tau ,\sigma \neq 0$ in the pseudo-Galilean space $G_{1}^{4}$.  Then, the position vector and curvatures of a vector satisfy 
differential equation of fourth order.%
\begin{equation*}
\left\{ \frac{1}{\sigma \varepsilon _{2}}\left\{ \left( \frac{1}{\tau
\varepsilon _{2}}\left( \frac{T^{\prime }}{\kappa \varepsilon _{1}}\right)
^{\prime }\right) ^{\prime }+\tau \varepsilon _{2}\frac{T^{\prime }}{\kappa
\varepsilon _{1}}\right\} \right\} ^{\prime }+\frac{\sigma }{\tau }\left( 
\frac{T^{\prime }}{\kappa \varepsilon _{1}}\right) ^{\prime }=0
\end{equation*}

\begin{proof}
Let $\alpha :I\subset IR\rightarrow G_{1}^{4}$ be an arbitrary admissible
curve $\kappa ,\tau ,\sigma \neq 0$ in $G_{1}^{4}$ and from Frenet equations
(2.11), we can write%
\begin{equation}
N=\frac{T^{\prime }}{\kappa \varepsilon _{1}};B_{1}=\frac{1}{\tau
\varepsilon _{2}}\left( \frac{T^{\prime }}{\kappa \varepsilon _{1}}\right)
^{\prime },  \tag{3.13}
\end{equation}%
\begin{equation}
B_{2}=\frac{1}{\sigma \varepsilon _{3}}\left( B_{1}^{\prime }+\tau
\varepsilon _{2}N\right) .  \tag{3.14}
\end{equation}%

By differentiating $B_{1}$ with respect to $s$ and substituting in (3.14),
we find%
\begin{equation}
B_{2}=\frac{1}{\sigma \varepsilon _{3}}\left( \left( \frac{1}{\tau
\varepsilon _{2}}\left( \frac{T^{\prime }}{\kappa \varepsilon _{1}}\right)
^{\prime }\right) ^{\prime }+\tau \varepsilon _{2}\frac{T^{\prime }}{\kappa
\varepsilon _{1}}\right) .  \tag{3.15}
\end{equation}%

Similarly, by taking the differentiation of (3.15) and equalize with (2.11),
we find%
\begin{equation}
\left\{ \frac{1}{\sigma \varepsilon _{2}}\left\{ \left( \frac{1}{\tau
\varepsilon _{2}}\left( \frac{T^{\prime }}{\kappa \varepsilon _{1}}\right)
^{\prime }\right) ^{\prime }+\tau \varepsilon _{2}\frac{T^{\prime }}{\kappa
\varepsilon _{1}}\right\} \right\} ^{\prime }+\frac{\sigma }{\tau }\left( 
\frac{T^{\prime }}{\kappa \varepsilon _{1}}\right) ^{\prime }=0.  \tag{3.16}
\end{equation}
\end{proof}
\end{theorem}

\begin{theorem}
Let $\alpha :I\rightarrow G_{1}^{4}$ be an arbitrary rectifying
curve in $G_{1}^{4}$ given by $\alpha (t)=\rho (t)\beta (t)$, where $\rho
(t) $ is an arbitrary positive function. Then, if $\alpha $ is a rectifying
curve, $\beta (t)$ is a unit speed rectifying curve and 
\begin{equation*}
0=\rho \left( \frac{\rho ^{\prime }}{v}\right) ^{\prime }(s^{2}+cs+d)+\left( 
\frac{\rho ^{\prime }}{v}+\left( \frac{\rho }{v}\right) ^{\prime }+\frac{%
\rho }{v}\right) (s+c^{\ast \ast }),
\end{equation*}%
where $v=\sqrt{\rho ^{\prime 2}+\rho ^{2}},c,d,c^{\ast \ast }\in 
\mathbb{R}
.$

\begin{proof}
Let $\alpha $ be an arbitrary admissible curve in $G_{1}^{4}$ given as 
\begin{equation*}
\alpha (t)=\rho (t)\beta (t)
\end{equation*}%
and $\rho (t)$ be an arbitrary positive function and $\beta (t)$ be a unit
speed curve. By taking the derivative of the previous curve with respect to $%
t$, we find 
\begin{equation}
\alpha ^{\prime }(t)=\rho ^{\prime }(t)\beta (t)+\rho (t)\beta ^{\prime }(t).
\tag{3.17}
\end{equation}%

Therefore, the tangent vector of $\alpha $ is given as follows%
\begin{equation}
T=\frac{\rho ^{\prime }}{v}\beta +\frac{\rho }{v}\beta ^{\prime }  \tag{3.18}
\end{equation}
where 
\begin{equation*}
v^{2}=\rho ^{\prime 2}+\rho ^{2}.
\end{equation*}

If we take the derivative with respect to $t$ again in (3.18) and using
(2.11), we find 
\begin{equation}
\varepsilon _{1}\kappa N=\left( \frac{\rho ^{\prime }}{v}\right) ^{\prime
}\beta +\left( \frac{\rho ^{\prime }}{v}+\left( \frac{\rho }{v}\right)
^{\prime }\right) \beta ^{\prime }+\frac{\rho }{v}\beta ^{\prime \prime }. 
\tag{3.19}
\end{equation}%

Then, we can write the following expressions for the curve $\beta $ for the
Frenet frame $\left\{ t_{\beta },n_{\beta },b_{1\beta },b_{2\beta }\right\} $
in $G_{1}^{4}$%
\begin{equation}
t_{\beta }^{\prime }=\varepsilon _{1}kn_{\beta };n_{\beta }^{\prime
}=\varepsilon _{2}\varkappa b_{1\beta };b_{1\beta }^{\prime }=-\varepsilon
_{2}\varkappa n_{\beta }+\varepsilon _{2}\sigma _{\beta }b_{2\beta
};b_{2\beta }^{\prime }=-\varepsilon _{2}\sigma _{\beta }b_{1\beta }. 
\tag{3.20}
\end{equation}%

We assume that $\beta ^{\prime }=t_{\beta }$ and decomposition of $%
\beta ^{\prime \prime }$ with respect to the frame $\left\{ t_{\beta
},n_{\beta },b_{1\beta },b_{2\beta }\right\} $ reads 
\begin{equation}
\beta ^{\prime \prime }=a_{1}t_{\beta }+a_{2}n_{\beta }+a_{3}b_{1\beta
}+a_{4}b_{2\beta }  \tag{3.21}
\end{equation}%
and inner products become as it follows that 
\begin{eqnarray*}
\left\langle \beta ^{\prime \prime },t_{\beta }\right\rangle
&=&a_{1}=1;\left\langle \beta ^{\prime \prime },n_{\beta }\right\rangle
=\varepsilon _{1}a_{2};\left\langle \beta ^{\prime \prime },b_{1\beta
}\right\rangle =\varepsilon _{2}a_{3};\left\langle \beta ^{\prime \prime
},b_{2\beta }\right\rangle =\varepsilon _{3}a_{4} \\
\left\langle \beta ^{\prime \prime },n_{\beta }\right\rangle &=&\varepsilon
_{1}a_{2}\rightarrow \left\langle t_{\beta }^{\prime },n_{\beta
}\right\rangle =\left\langle \varepsilon _{1}kn_{\beta },n_{\beta
}\right\rangle =\varepsilon _{1}^{2}k=\varepsilon _{1}a_{2}\rightarrow
a_{2}=k\varepsilon _{1};a_{4},a_{3}=0
\end{eqnarray*}%
so the equation (3.21) becomes 
\begin{equation}
\beta ^{\prime \prime }=t_{\beta }+k\varepsilon _{1}n_{\beta }.  \tag{3.22}
\end{equation}%

By using (3.22) into (3.19), we can have 
\begin{equation}
\varepsilon _{1}\kappa N=\left( \frac{\rho ^{\prime }}{v}\right) ^{\prime
}\beta +\left( \frac{\rho ^{\prime }}{v}+\left( \frac{\rho }{v}\right)
^{\prime }+\frac{\rho }{v}\right) t_{\beta }+\varepsilon _{1}k\frac{\rho }{v}%
n_{\beta }.  \tag{3.23}
\end{equation}%

For the last equation, when the $\alpha $ curve is considered to be a
rectifying curve, if both sides of this equation are subjected to the inner
product with $\alpha $, the last equation is obtained as 
\begin{equation}
\varepsilon _{1}\kappa \left\langle N,\alpha \right\rangle =\left( \frac{%
\rho ^{\prime }}{v}\right) ^{\prime }\left\langle \beta ,\alpha
\right\rangle +\left( \frac{\rho ^{\prime }}{v}+\left( \frac{\rho }{v}%
\right) ^{\prime }+\frac{\rho }{v}\right) \left\langle t_{\beta },\alpha
\right\rangle +\varepsilon _{1}k\frac{\rho }{v}\left\langle n_{\beta
},\alpha \right\rangle.  \tag{3.24}
\end{equation}%

Since $\beta $ is a rectifying curve, we can say that $\left\langle N,\alpha
\right\rangle =0.$ 
\begin{equation}
\left\langle \beta ,\alpha \right\rangle =\rho \left\Vert \beta \right\Vert ^{2};\left\langle t_{\beta
},\alpha \right\rangle =\rho \left\langle t_{\beta },\beta \right\rangle
=\mu _{1};\left\langle n_{\beta },\alpha \right\rangle =\rho \left\langle
n_{\beta },\beta \right\rangle =\varepsilon _{1}\mu _{2},  \tag{3.25}
\end{equation}%
from (3.25), by making necessary calculation, we have%
\begin{equation}
0=\left( \frac{\rho ^{\prime }}{v}\right) ^{\prime }\rho \left\Vert \beta
\right\Vert ^{2}+\left( \frac{\rho ^{\prime }}{v}+\left( \frac{\rho }{v}%
\right) ^{\prime }+\frac{\rho }{v}\right) \mu _{1}+k\frac{\rho }{v}\mu _{2}.
\tag{3.26}
\end{equation}%

Here, if the $\beta $ is considered to be a rectifying curve, that is, if $%
\mu _{2}=0$ and $\left\Vert \beta \right\Vert ^{2}=2\int \mu _{1}ds$, the
last equation is written as%
\begin{equation}
0=\rho \left( \frac{\rho ^{\prime }}{v}\right) ^{\prime }(s^{2}+cs+d)+\left( 
\frac{\rho ^{\prime }}{v}+\left( \frac{\rho }{v}\right) ^{\prime }+\frac{%
\rho }{v}\right) (s+c^{\ast \ast }).  \tag{3.27}
\end{equation}
\end{proof}

\begin{definition}
A curve $\alpha $ in the 4-dimensional pseudo-Galilean space $G_{1}^{4}$ is
called an admissible slant helix if the normal lines of $\alpha $ make a
constant angle with a fixed direction.
\end{definition}
\end{theorem}

\begin{theorem}
Let $\alpha $ be an admissible slant helix with non-vanishing curvatures in $%
G_{1}^{4}.$ Then,

1) there is relation among curvatures of \ $\alpha $ as 
\begin{equation*}
\frac{\tau }{\sigma }=\frac{c}{\cos w}=\text{constant, }w\neq k\frac{\pi }{2}%
,k\in 
\mathbb{Z}
.
\end{equation*}

2) fixed direction of this slant helix can be written as%
\begin{equation*}
U=\cos wN+\frac{\tau }{\sigma }\cos wB_{2}.
\end{equation*}
\end{theorem}

\begin{proof}
From previous definition of slant helix, we can write%
\begin{equation}
\left\langle N,U\right\rangle _{G_{1}^{4}}=\cos w,  \tag{3.28}
\end{equation}%
where $U$ is a constant vector and $w\neq k\frac{\pi }{2}$, $ k\in 
\mathbb{Z}$ is a constant
angle. If we take derivative of (3.28) and consider the Frenet equations
(2.11), we get%
\begin{equation}
\left\langle \varepsilon _{2}\tau B_{1},U\right\rangle _{G_{1}^{4}}=0. 
\tag{3.29}
\end{equation}%

The equation (3.29) implies that $B_{1}\perp U$. Therefore can compose $U$ as%
\begin{equation}
U=u_{1}T+u_{2}N+u_{3}B_{2}.  \tag{3.30}
\end{equation}%

Differentiating (3.30), we easily obtain%
\begin{equation*}
\frac{\tau }{\sigma }=\frac{c}{\cos w}=\text{constant},
\end{equation*}%
\begin{equation*}
u_{1}=0,u_{2}=\cos w=const.,u_{3}=\frac{\tau }{\sigma }\cos w=\text{constant},
\end{equation*}%
where $\frac{\tau }{\sigma }=constant.$ Thus, the proof is completed.
\end{proof}

\begin{definition}
A curve $\alpha $ in pseudo-Galilean space $G_{1}^{4}$ is called an
admissible 3-type slant helix if the trinormal lines of make a constant
angle with a fixed direction.

\begin{theorem}
Let $\alpha $ be an admissible 3-type slant helix with non-vanishing
curvatures in $G_{1}^{4}.$ Then,

1) there is a relation among curvatures of $\alpha $ as%
\begin{equation*}
\cos \phi \left( \frac{\sigma }{\tau }\right) ^{\prime }=a\kappa ,a\in 
\mathbb{R}
.
\end{equation*}

2) fixed direction of this helix can be written as%
\begin{equation*}
U=bT+\frac{\sigma }{\tau }\cos \phi N+\cos \phi B_{2},
\end{equation*}

where $\phi \neq k\frac{\pi }{2},k\in 
\mathbb{Z}
$ is a constant angle, $a$ and $b$ are real numbers.

\begin{proof}
From previous definition, we write%
\begin{equation}
\left\langle B_{2},U\right\rangle _{G_{1}^{4}}=\cos \phi,  \tag{3.31}
\end{equation}%
where $U$ is a constant vector and $\phi \neq k\frac{\pi }{2}(k\in 
\mathbb{Z}
$) is a constant angle. If we take derivative of (3.31) and by considering the
Frenet equations (2.11), we get%
\begin{equation}
\left\langle -\varepsilon _{2}\sigma B_{1},U\right\rangle
_{G_{1}^{4}}=0\Rightarrow B_{1}\perp U  \tag{3.32}
\end{equation}%
and we compose constant vector $U$ as%
\begin{equation}
U=v_{1}T+v_{2}N+v_{3}B_{2}.  \tag{3.33}
\end{equation}%

Differentiating both sides of (3.33), we easily obtain%
\begin{equation*}
v_{1}=b,v_{2}=\frac{\sigma }{\tau }\cos \phi ,v_{3}=\cos \phi \text{ and }%
\cos \phi \left( \frac{\sigma }{\tau }\right) ^{\prime }=a\kappa ,a\in 
\mathbb{R}%
\end{equation*}%
which completes the proof.
\end{proof}
\end{theorem}
\end{definition}

\begin{corollary}
If an admissible 3-type slant helix with non-vanishing curvatures $\alpha $
is a $W$-curve in $G_{1}^{4},$ then fixed direction of $\alpha $ as follow:%
\begin{equation}
U=\frac{\sigma }{\tau }\cos \phi N+\cos \phi B_{2} \tag{3.34}
\end{equation}
where $\frac{\sigma }{\tau }$ and $\cos \phi $ are constants.
\end{corollary}

\begin{theorem}
There isn't any admissible curve in $G_{1}^{4}$ such that the bi-normal
lines of it make a constant angle with a fixed direction.

\begin{proof}
Let us assume that there is a curve which satisfies hypothesis of the
theorem. Thus we can write%
\begin{equation}
\left\langle U,B_{1}\right\rangle _{G_{1}^{4}}=\cos \theta,   \tag{3.35}
\end{equation}%
where $U$ is constant vector and $\theta $ is a constant angle. Then, by differentiating both sides of (3.35) and considering Frenet equations
(2.11), we obtain%
\begin{equation*}
\left\langle U,-\varepsilon _{2}\tau N+\varepsilon _{3}\sigma
B_{2}\right\rangle _{G_{1}^{4}}=0
\end{equation*}%
which implies that $N\perp U$ and $B_{2}\perp U.$ Therefore, we can compose $%
U$ as%
\begin{equation}
U=w_{1}T+w_{2}B_{1}.  \tag{3.36}
\end{equation}%

Differentiating (3.36), we easily have $w_{1}=w_{2}=0,$ which is a
contradiction.
\end{proof}
\end{theorem}

\begin{definition}
Let $c$ and $r>0$ given and $f(s)=\left( \alpha -c\right) ^{2}.$ Then, $%
\alpha $ has $j$ order spherical contact with sphere of radius $r$ and
center $c$ at $s=s_{0}$ if 
\begin{equation*}
f(s_{0})=r^{2},f^{\prime }(s_{0})=f^{\prime \prime }(s_{0})=...=f^{\left(
j\right) }(s_{0})=0.
\end{equation*}
\end{definition}

Recall that, in the Euclidean space, for an arbitrary curve $\alpha $ lies
on a sphere with center $c$, then $\left( \alpha -c\right) ^{2}$ is constant
and so we consider the previous definition of contact \cite{14}.

\begin{theorem}
Let $\alpha (s)$ $=\left( s,y(s),z(s),w(s)\right) $ be an admissible curve
parametrized by arclength $s$ in $G_{1}^{4}$. If $\alpha $ lies on the
pseudo-Galilean sphere in $G_{1}^{4}.$ Then, the center is given as%
\begin{equation*}
c(s)=\alpha (s)+\rho (s)N(s)+\varepsilon _{1}\frac{\rho ^{\prime }}{\tau (s)}%
B_{1}(s)+\frac{\varepsilon _{1}\varepsilon _{2}}{\sigma (s)}\left( \tau
(s)\rho (s)+\left( \frac{\rho ^{\prime }}{\sigma }\right) ^{\prime }\right)
B_{2}(s)
\end{equation*}
where $\rho (s)=\frac{1}{\kappa (s)}.$
\end{theorem}

\begin{proof}
Let $\alpha (s)$ $=\left( s,y(s),z(s),w(s)\right) $ be a curve parametrized
by arclength $s$ in $G_{1}^{4}$ with the Frenet-Serret equations (2.11).
From previous definition, we can write%
\begin{equation}
f(s)=\left\langle c(s)-\alpha (s),c(s)-\alpha (s)\right\rangle
_{G_{1}^{4}}=\pm r^{2}  \tag{3.37}
\end{equation}%
and%
\begin{equation*}
f^{\prime }(s)=f^{\prime \prime }(s)=f^{\prime \prime \prime }(s)=f^{\left(
\imath v\right) }(s)=0.
\end{equation*}%

Thus, if we take derivative of (3.37), we get%
\begin{equation}
\left\langle -T(s),c(s)-\alpha (s)\right\rangle _{G_{1}^{4}}=0  \tag{3.38}
\end{equation}%
and again by differentiating of (3.38), we have%
\begin{equation}
-\varepsilon _{1}\kappa (s)\left\langle N(s),c(s)-\alpha (s)\right\rangle
_{G_{1}^{4}}+1=0.  \tag{3.39}
\end{equation}%

We may compose the vector%
\begin{equation}
c(s)-\alpha (s)=c_{1}(s)T(s)+c_{2}(s)N(s)+c_{3}(s)B_{1}(s)+c_{4}(s)B_{2}(s).
\tag{3.40}
\end{equation}%

From (3.38), we easily obtain $c_{1}(s).$ Using (3.39) in (3.40), we have%
\begin{equation}
c_{2}(s)=\frac{1}{\kappa (s)}.  \tag{3.41}
\end{equation}%

Differentiating (3.40), we obtain%
\begin{equation*}
c_{3}(s) =\varepsilon _{1}\frac{1}{\tau (s)}\left( \frac{1}{\kappa (s)}%
\right) ^{\prime },  \tag{3.42} \\
\end{equation*}%
\begin{equation*}
c_{4}(s) =\frac{\varepsilon _{1}\varepsilon _{2}}{\sigma (s)}\left( \frac{%
\tau (s)}{\kappa (s)}+\left( \frac{1}{\tau (s)}\left( \frac{1}{\kappa (s)}%
\right) ^{\prime }\right) ^{\prime }\right) .  \tag{3.43}
\end{equation*}%

If we consider the equations $c_{1}(s),c_{2}(s),c_{3}(s)$ and $c_{4}(s)$ in
(3.40) and denote $\rho =\frac{1}{\kappa },$ we get 
\begin{equation}
c(s)=\alpha (s)+\frac{1}{\kappa }N(s)+\varepsilon _{1}\frac{1}{\tau }\left( 
\frac{1}{\kappa }\right) ^{\prime }B_{1}(s)+\frac{\varepsilon
_{1}\varepsilon _{2}}{\sigma }\left( \frac{\tau }{\kappa }+\left( \frac{1}{%
\tau }\left( \frac{1}{\kappa }\right) ^{\prime }\right) ^{\prime }\right)
B_{2}(s).  \tag{3.44}
\end{equation}%

Thus, the proof is completed.
\end{proof}

\begin{theorem}
Let $\alpha (s)$ $=\left( s,y(s),z(s),w(s)\right) $ be an admissible curve
parametrized by arclength $s$ in $G_{1}^{4}$. If $\alpha $ lies on the
pseudo-Galilean sphere in $G_{1}^{4}.$ Then, the radius of the
pseudo-sphere satisfies%
\begin{equation}
r^{2}=\left\vert \varepsilon _{1}\rho ^{2}+\varepsilon _{2}\left( \frac{\rho
^{\prime }}{\tau }\right) ^{2}+\varepsilon _{3}\frac{1}{\sigma ^{2}}\left(
\tau \rho +\left( \frac{\rho ^{\prime }}{\tau }\right) ^{\prime }\right)
^{2}\right\vert .  \tag{3.45}
\end{equation}

\begin{proof}
The proof is can be easily obtained by the taking norm of both sides of
(3.44).
\end{proof}
\end{theorem}

\begin{theorem}
Let $\alpha (s)$ $=\left( s,y(s),z(s),w(s)\right) $ be an admissible curve
parametrized by arclength $s$ in $G_{1}^{4}$. If $\alpha $ lies on the
pseudo Galilean sphere in $G_{1}^{4}.$ If $\alpha $ lies on the
pseudo-Galilean sphere $S_{1}^{3}(c,r)$, then the curvature functions of\ $%
\alpha $ satisfy the differential equation%
\begin{equation}
2\varepsilon _{1}\frac{\rho ^{\prime }}{\tau }\left( \tau \rho +\varepsilon
_{2}\left( \frac{\rho ^{\prime }}{\tau }\right) ^{\prime }\right)
+\varepsilon _{3}\left( \frac{1}{\sigma ^{2}}\left( \tau \rho +\left( \frac{%
\rho ^{\prime }}{\tau }\right) ^{\prime }\right) ^{2}\right) ^{\prime }=0. 
\tag{3.46}
\end{equation}
\end{theorem}

\begin{proof}
The proof can be obtained by the differentiation of the equation (3.45).
\end{proof}

\begin{definition}
An admissible curve $\alpha (s)$ in the pseudo Galilean 4-space $G_{1}^{4}$
is called osculating curve if it has no component in the first bi-normal
direction or the second bi-normal direction, in other words $\left\langle
\alpha (s),B_{1}\right\rangle _{G_{1}^{4}}=0$ or $\left\langle \alpha
(s),B_{2}\right\rangle _{G_{1}^{4}}=0$ are satisfied. Therefore, the position vector satisfies the parametric equations, respectively%
\begin{equation}
\alpha (s)=\varrho _{1}(s)T(s)+\varrho _{2}(s)N(s)+\varrho _{3}(s)B_{2}(s) 
\tag{3.47}
\end{equation}%
and 
\begin{equation}
\alpha (s)=\varrho ^{1}(s)T(s)+\varrho ^{2}(s)N(s)+\varrho ^{3}(s)B_{1}(s), 
\tag{3.48}
\end{equation}%
for some differentiable functions $\varrho _{i}(s),\varrho ^{i}(s),1\leq
i\leq 3.$
\end{definition}

\begin{theorem}
Let $\alpha =\alpha (s)$ be osculating curve of type 1 with curvatures $%
\kappa,\tau,\sigma \neq 0$ in the pseudo-Galilean space $%
G_{1}^{4} $. Then, the position vector of osculating curve of type 1 can be
written as 
\begin{equation}
\alpha (s)=\left( s+c_{13}\right) T(s)+c_{14}\left( \frac{\sigma (s)}{\tau
(s)}\right) N(s)+c_{14}(s)B_{2}(s)  \tag{3.49}
\end{equation}%
and the following equation is satisfied 
\begin{equation*}
\left( s+c_{13}\right) \kappa (s)+c_{14}\left( \frac{\sigma (s)}{\tau (s)}%
\right) =0,
\end{equation*}
where $c_{13},c_{14}\in 
\mathbb{R}
.$
\end{theorem}

\begin{proof}
From definition osculating curve of type 1 with curvatures $\kappa,\tau,\sigma\neq 0$ given in $G_{1}^{4},$ by differentiating Eq. (3.47)
with respect to arc length parameter $s$ and using the Serret-Frenet
equations (2.11), we get%
\begin{equation*}
\alpha ^{\prime }(s)=\varrho _{1}^{\prime }(s)T+(\varrho _{1}(s)\varepsilon
_{1}\kappa +\varrho _{2}^{\prime }(s))N+\varepsilon _{2}(\varrho _{2}(s)\tau
-\varrho _{3}(s)\sigma )B_{1}+\varrho _{3}^{\prime }(s)B_{2}
\end{equation*}%
it follows that 
\begin{equation*}
\varrho _{1}^{\prime }(s)=1;\varrho _{1}(s)\kappa +\varepsilon _{1}\varrho
_{2}^{\prime }(s)=0;\varrho _{2}(s)\tau -\varrho _{3}(s)\sigma
=0;\varepsilon _{3}\varrho _{3}^{\prime }(s)=0.
\end{equation*}%

By solving previous equations together, we get 
\begin{equation*}
\varrho _{1}(s)=s+c_{13};\varrho _{3}(s)=c_{14}=const.;\varrho _{2}(s)=c_{14}%
\frac{\sigma }{\tau }.
\end{equation*}%

Also, from the second equation, we can write the following equation 
\begin{equation*}
\left( s+c_{13}\right) \kappa (s)+c_{14}\left( \frac{\sigma (s)}{\tau (s)}%
\right) =0.
\end{equation*}
\end{proof}

\begin{theorem}
Let $\alpha =\alpha (s)$ be osculating curve of type 2 with curvatures $%
\kappa ,\tau ,\sigma $ in the pseudo-Galilean space $G_{1}^{4}$. Then, the
position vector of osculating curve of type 2 can be written as

1) for $\sigma (s)\neq 0,$ 
\begin{equation}
\alpha (s)=\left( s+c_{15}\right) T(s)-\left( \varepsilon _{1}\int \left(
s+c_{15}\right) \kappa \left( s\right) ds+c_{16}.\right) N(s).  \tag{3.50}
\end{equation}%

2) for $\sigma (s)=0$ and $\tau ,\kappa =$constant$,$%
\begin{equation}
\alpha (s)=\left( s+c_{15}\right) T(s)+\left( -\varepsilon _{2}\frac{\varrho
_{3}^{\prime }}{\tau }\right) N(s)+\varrho _{3}B_{1}(s)  \tag{3.51}
\end{equation}%
where 
\begin{equation*}
\varrho _{3}(s)=a_{1}\sin \tau s+a_{2}\cos \tau s+\varepsilon
_{1}\varepsilon _{2}\frac{\kappa }{\tau }(s+c_{15}),c_{15},a_{1},a_{2}\in 
\mathbb{R}
.
\end{equation*}
\end{theorem}

\begin{proof}
From definition osculating curve of type 2 with curvatures $\kappa,\tau,\sigma\neq 0$ given in $G_{1}^{4},$ by differentiating Eq. (3.48)
with respect to arc length parameter $s$ and using the Serret-Frenet
equations (2.11), we get%
\begin{equation*}
T=\varrho _{1}^{\prime }T+(\varrho _{1}\varepsilon _{1}\kappa +\varrho
_{2}^{\prime }-\varepsilon _{2}\tau \varrho _{3})N+(\varrho _{3}{}^{\prime
}+\varepsilon _{2}\varrho _{2}\tau )B_{1}+\varepsilon _{3}\varrho _{3}\sigma
B_{2},
\end{equation*}%
by using previous equality and for $\sigma \neq 0$ similar to the previous
proof, the following functions are obtained 
\begin{equation*}
\varrho _{1}(s)=s+c_{15};\varrho _{3}(s)=0;\varrho _{2}(s)=-\varepsilon
_{1}\int \left( s+c_{15}\right) \kappa \left( s\right) ds+c_{16}.
\end{equation*}%

Also, for $\sigma (s)=0$ and $\tau ,\kappa =$constant$,$ we get 
\begin{equation*}
\varrho _{3}^{\prime \prime }+\tau ^{2}\varrho _{3}=\tau \kappa \varepsilon
_{1}\varepsilon _{2}(s+c_{15})
\end{equation*}%
and solving the previous differential equation, we have 
\begin{equation*}
\varrho _{3}(s)=a_{1}\sin \tau s+a_{2}\cos \tau s+\varepsilon
_{1}\varepsilon _{2}\frac{\kappa }{\tau }(s+c_{15}),c_{15},a_{1},a_{2}\in 
\mathbb{R}%
\end{equation*}%
and from the equation $\varepsilon _{2}\left( \varrho _{3}(s)^{\prime
}+\varepsilon _{2}\varrho _{2}(s)\tau \right) =0$, we can obtain the
equation $\varrho _{2}(s)$.
\end{proof}

\begin{definition}
An admissible curve $\alpha (s)$ in the pseudo Galilean 4-space $G_{1}^{4}$
is called rectifying curve if it has no component in the normal direction and $\left\langle \alpha (s),N\right\rangle _{G_{1}^{4}}=0$. Then, the position vector satisfies the parametric equation
\begin{equation}
\alpha (s)=\chi _{1}(s)T(s)+\chi _{2}(s)B_{1}(s)+\chi _{3}(s)B_{2}(s),
\tag{3.52}
\end{equation}%
for some differentiable functions $\chi _{i}(s),1\leq i\leq 3.$
\end{definition}

\begin{theorem}
Let $\alpha =\alpha (s)$ be rectifying curve with curvatures $\kappa
(s),\tau (s),\sigma (s)\neq 0$ in the pseudo-Galilean space $G_{1}^{4}$.
Then, the position vector of rectifying curve can be written as 
\begin{equation}
\alpha (s)=\left( s+c_{15}\right) T(s)+\varepsilon _{1}\varepsilon
_{2}\left( s+c_{15}\right) \frac{\kappa }{\tau }B_{1}(s)+\frac{\varepsilon
_{1}}{\sigma }\left( \left( s+c_{15}\right) \frac{\kappa }{\tau }\right)
^{\prime }B_{2}(s),  \tag{3.53}
\end{equation}
where $c_{15}\in 
\mathbb{R}
.$

\begin{proof}
From definition rectifying curve with curvatures $\kappa (s),\tau (s),\sigma
(s)\neq 0$ given in $G_{1}^{4},$ by differentiating Eq. (3.52) with respect
to arc length parameter $s$ and using the Serret-Frenet equations (2.11), we
get%
\begin{equation*}
T=\chi _{1}^{\prime }T+(\chi _{1}\varepsilon _{1}\kappa -\varepsilon
_{2}\tau \chi _{2})N+(\chi _{2}{}^{\prime }-\varepsilon _{2}\chi _{2}\sigma
)B_{1}+(\varepsilon _{3}\chi _{2}\sigma +\chi _{3}^{\prime })B_{2}.
\end{equation*}%

Hence, 
\begin{equation}
\chi _{1}^{\prime }=1;\chi _{1}\kappa -\varepsilon _{2}\varepsilon
_{1}\tau \chi _{2}=0;\varepsilon _{2}\chi _{2}^{\prime }-\chi
_{2})\sigma =0;\chi _{2}\sigma +\varepsilon _{3}\chi _{3}^{\prime
}=0.  \tag{3.54}
\end{equation}%

By solving Eq. (3.54) together, we get 
\begin{equation*}
\chi _{1}(s)=s+c_{15};\varrho _{3}(s)=\frac{\varepsilon _{1}}{\sigma }\left(
\left( s+c_{15}\right) \frac{\kappa }{\tau }\right) ^{\prime };\varrho
_{2}(s)=\varepsilon _{1}\varepsilon _{2}\left( s+c_{15}\right) \frac{\kappa 
}{\tau }.
\end{equation*}
\end{proof}

\begin{definition}
An admissible curve $\alpha (s)$ in the pseudo Galilean 4-space $G_{1}^{4}$
is called normal curve if it has no component in the tangent direction and $\left\langle \alpha (s),T\right\rangle _{G_{1}^{4}}=0$. Then, the position
vector satisfies the parametric equation%
\begin{equation}
\alpha (s)=n_{1}(s)N(s)+n_{2}(s)B_{1}(s)+n_{3}(s)B_{2}(s),  \tag{3.55}
\end{equation}%
for some differentiable functions $n_{i}(s),1\leq i\leq 3.$

\begin{theorem}
There is no normal curve with curvatures $\kappa (s),\tau (s),\sigma (s)\neq
0$ in the pseudo-Galilean space $G_{1}^{4}.$

\begin{proof}
If we consider the equation in (3.55) and by taking the derivative this
equation, we can easily see that 
\begin{equation*}
\left\langle \alpha (s),T\right\rangle _{G_{1}^{4}}=\varepsilon
_{1}(n_{1}^{\prime }-\varepsilon _{2}\tau n_{2})^{2}+\varepsilon _{2}\left(
\varepsilon _{2}\tau n_{1}+n_{1}^{\prime }-\varepsilon _{2}\sigma
n_{3}\right) ^{2}+\varepsilon _{3}\left( \varepsilon _{3}\sigma
n_{2}+n_{3}^{\prime }\right)^{2} \neq 0.
\end{equation*}%

Therefore, from the definition, we can say that there is no normal curve in $%
G_{1}^{4}$.
\end{proof}
\end{theorem}
\end{definition}
\end{theorem}

\subsection*{Acknowledgement}
The authors would like to thank referees for their valuable suggestions.


\begin{thebibliography}{99}
\bibitem{1} Akb\i y\i k M., Y\"{u}ce S., 4-dimensional pseudo-Galilean
geometry, Cumhuriyet Sci. J., 42(4), 890-905, 2021.

\bibitem{2} Ali A., Turgut M., Position vector of a time-like slant helix in
Minkowski 3-space, J. Math. Anal. Appl., 365, 559-569, 2010.

\bibitem{3} Ali A.T., Position vectors of curves in the Galilean space $%
G^{3} $, Matemati\v{c}ki Vesnik, 64(249), 200-210, 2012.

\bibitem{4} Almaz F., K\"{u}lahc\i\ M.A., A survey on magnetic curves in
2-dimensional lightlike cone, Malaya Journal of Matematik, 7(3), 477-485,
2019.

\bibitem{5} Almaz F., K\"{u}lahc\i\ M.A., On x-magnetic surfaces generated
by trajectory of $x-$magnetic curves in null cone, General Letters in
Mathematics, 5(2), 84-92, 2018.

\bibitem{6} Almaz F., K\"{u}lahc\i\ M.A., Involute-Evolute D-curves in
Minkowski 3-space $E_{1}^{3}$, Bol. Soc. Paran. Mat., 39(1), 147-156.

\bibitem{7} Almaz F., K\"{u}lahc\i\ M.A., The helix and slant helices
generated by non-degenerate curves in $M^{3}(\rho_{0})\subset $ $%
E_{2}^{4}$, Turkish Journal of Science and Technology, 16(1), 113-117, 2021.

\bibitem{8} Almaz F., K\"{u}lahc\i\ M.A., Non-null normal curves in the
semi-Euclidean space $E_{2}^{4}$, Mugla Journal of Science and Technology,
7(1), 137-140, 2021.

\bibitem{9} B\"{u}y\"{u}kk\"{u}t\"{u}k S., Ki\c{s}i I., Mishra V.N., \"{O}%
zturk G., Some characterizations of curves in Galilean 3-space $G^{3}$,
Facta Universitatis, Ser Maths Inform, 31(2), 503--12, 2016.

\bibitem{10} Divjak B., Curves in pseudo-Galilean Geometry, Annales Univ.
Budapest, 41, 117-128, 1998.

\bibitem{11} Erg\"{u}t M., \"{O}\u{g}renmi\c{s} A.O., Some characterizations
of a spherical curves in Galilean space, 1(2), 18-26, 2009.

\bibitem{12} Kalkan O.B., Position vector of a W-curve in the 4D Galilean
space, Facta Universitatis. Ser Maths Inform, 31(2), 485--92, 2016.

\bibitem{13} K\"{u}lahc\i\ M., Almaz F., New classifications for space-like
curves in the null cone, Prespacetime Journal, 7(15), 2002-2014, 2016.

\bibitem{14} Millman R.S., Parker G.D., Elements of Differential Geometry,
Prentice Hall Inc., Englewood Clifs, New Jersey, 1977.

\bibitem{15} \"{O}\u{g}renmi\c{s} A.O., Erg\"{u}t M., On the explicit
characterization of admissible curve in 3-dimensional pseudo-Galilean space,
J. Adv. Math. Stud., 2(1), 63-72, 2009.

\bibitem{16} \"{O}ztekin H., Normal and rectifying curves in Galilean space $%
G^{3}$, Proceedings of institute of applied mathematics, 98--109, 2016.

\bibitem{17} \"{O}ztekin H.B., Tatlipinar S., On some curves in Galilean
plane and 3-dimensional Galilean space, Journal of Dynamical Systems and
Geometric Theories, 10 (2), 189-196, 2012.

\bibitem{18} Tarla S., \"{O}ztekin H., On generalized Mannheim curves in the
equform Geometry of the Galilean 4-space, Turkish Journal of Science and
Technology, 16 (2), 197-204, 2021.

\bibitem{19} Y\i lmaz S., Spherical indicators of curves and
characterizations of some special curves in four dimensional Lorentzian
space $L^{4}$, Dissertation, Dokuz Eyl\"{u}l University, 2001.

\bibitem{20} Y\i lmaz S., Construction of the Frenet-Serret frame of a Curve
in 4D Galilean space and some applications, Int. J. of the Physical
Sciences, 5(8), 1284-1289, 2010.

\bibitem{21} Yoon D.W., Lee J.W., Lee C.W., Osculating curves in the
Galilean 4-space, Int J. Pure Appl. Maths., 100(4), 497--506, 2015.
\end{thebibliography}
\end{document}